\documentclass[12pt]{article}

\usepackage{amssymb,amsthm,amsmath}

\newcommand{\dd}{\mathrm{d}}
\newcommand{\E}{\mathbb{E}}

\newcommand{\n}[1]{\left\| #1 \right\|}

\DeclareMathOperator{\sgn}{sgn}

\newtheorem{thm}{Theorem}
\newtheorem{lem}[thm]{Lemma}

\newtheorem{prop}[thm]{Proposition}
\newtheorem{cor}[thm]{Corollary}

\theoremstyle{definition}
\newtheorem{rem}{Remark}
\newtheorem{remark}[rem]{Remark}
\newtheorem{dfn}{Definition}

\newcommand{\1}{\textbf{1}}
\newcommand{\wt}[1]{\widetilde{#1}}
\def\Pr{\mathbb P}
\def\er{\mathbb R}
\def\Ex{\mathbb E}
\def\Var{\mathrm{Var}}
\def\ve{\varepsilon}

\title{A note on suprema of canonical processes based on random variables with regular moments}
\author{Rafa{\l} Lata{\l}a\thanks{Research supported by the NCN grant DEC-2012/05/B/ST1/00412} and Tomasz Tkocz}
\date{}

\begin{document}

\maketitle

\begin{abstract}
We derive two-sided bounds for expected values of suprema of canonical processes based on random variables
with moments growing regularly. We also discuss a Sudakov-type minoration principle for canonical
processes.
\end{abstract}

\section{Introduction and Main Results}

In many problems arising in probability theory and its applications one needs to estimate the supremum of a stochastic
process. In particular it is very useful to be able to find two-sided bounds for the mean of the
supremum. The modern approach to this challenge is based on the chaining methods, see monograph \cite{Talnew}.

In this note we study the class of so-called \emph{canonical processes} of the form
$X_t=\sum_{i=1}^{\infty}t_iX_i$, where $X_i$ are independent random variables. If $X_i$ are \emph{standardized}, 
i.e. have mean zero and variance one, then the above series converges a.s.\  for $t\in \ell^2$ and we may try 
to estimate $\Ex\sup_{t\in T}X_t$ for $T\subset \ell^2$. 
To avoid measurability questions we either assume that the index set $T$ is countable or define
in a general situation
\[
\Ex\sup_{t\in T}X_t:=\sup\left\{\Ex\sup_{t\in F}X_t\colon\ F\subset T \mbox{ finite }\right\}.
\]
It is also more convenient to work with the quantity  $\Ex\sup_{s,t\in T}(X_t-X_s)$ rather than $\Ex\sup_{t\in T}X_t$.
Observe however that if the set $T$ or the variables $X_i$ are symmetric then
\[
\Ex\sup_{s,t\in T}(X_s-X_t)=\Ex\sup_{s\in T}X_s+\Ex\sup_{t\in T}(-X_t)=2\Ex\sup_{t\in T}X_t.
\]

In the case when $X_i$ are i.i.d.\ ${\mathcal N}(0,1)$ r.v.s, $X_t$ is the canonical Gaussian process. Moreover,  any 
centered separable Gaussian process has the Karhunen-Lo\`eve representation of such form (see e.g. Corollary 5.3.4 in \cite{Marcus}). In the Gaussian case 
the behaviour of the supremum
of the process is related to the geometry of the metric space $(T,d_2)$, where $d_2$ is the $\ell^2$-metric
$d(s,t)=(\Ex|X_s-X_t|^2)^{1/2}$. The celebrated Fernique-Talagrand majorizing measure bound 
(cf. \cite{Fe,Ta_reg}) may be expressed in the form
\[
\frac{1}{C}\gamma_2(T)\leq \Ex\sup_{t\in T}X_t\leq C\gamma_2(T),
\]
where here and in the sequel $C$ denotes a universal constant,
\[
\gamma_2(T):=\inf\sup_{t\in T}\sum_{n=0}^{\infty}2^{n/2}\Delta_2(A_n(t)),
\]
the infimum runs over all admissible sequences of partitions $({\mathcal A}_n)_{n\geq 0}$ of the
set $T$, $A_n(t)$ is the unique set in ${\mathcal A}_n$ which contains $t$, and $\Delta_2$ denotes the $\ell^2$-diameter.
An increasing sequence of partitions $({\mathcal A}_n)_{n\geq 0}$ of $T$ is called \emph{admissible} if ${\cal A}_0=\{T\}$ 
and $|{\cal A}_n|\leq N_n:=2^{2^n}$ for $n\geq 1$.

In \cite{Ta_can} Talagrand derived two-sided bounds for suprema of the canonical processes based on i.i.d.\ symmetric
r.v.s $X_i$ such that $\Pr(|X_i|> t)=\exp(-|t|^p)$, $1\leq p<\infty$. This result was later extended in \cite{La_sup}
to the case of variables with (not too rapidly decreasing) log-concave tails, i.e.\ to the case when  
$X_i$ are symmetric, independent, $\Pr(|X_i|\geq t)=\exp(-N_i(t))$, $N_i\colon [0,\infty)\rightarrow [0,\infty)$
are convex and $N_i(2t)\leq \gamma N_i(t)$ for $t>0$ and some constant $\gamma$. 
The aim of this note is to find two-sided bounds for suprema
for a more general class of canonical processes. 

For a general process $(X_t)_{t\in T}$ one needs to study a family of metrics instead of a single one. We
define
\[
d_p(s,t):=\|X_s-X_t\|_p,\quad p\geq 1,\ s,t\in T,
\]
where for a real random variable $Y$ and $p\geq 1$, $\|Y\|_p:=(\Ex|Y|^p)^{1/p}$ denotes the $L^p$-norm of $Y$.
Following ideas of Talagrand, we define the functional
\[
\gamma_X(T):=\inf\sup_{t\in T}\sum_{n=0}^{\infty}\Delta_{2^n}(A_n(t)),
\]
where $\Delta_{p}$ denotes the diameter with respect to the  distance $d_p$ and as in the case of the $\gamma_2$-functional
the infimum runs over all admissible sequences of partitions $({\cal A}_n)$ of the set $T$.

It is not hard to show (as it was noted independently by Mendelson and the first named author, c.f.
\cite[Exercise 2.2.25]{Talnew}) that for \emph{any} 
process $(X_t)_{t\in T}$,
\begin{equation}
\label{eq:upper}
\Ex\sup_{s,t\in T} (X_s-X_t)\leq C\gamma_X(T).
\end{equation}

To reverse bound \eqref{eq:upper} we need some regularity assumptions. We express them for canonical processes
in terms of moments growth of variables $X_i$. It is easy to check that
for a symmetric variable $Y$ with a log-concave tail $\exp(-N(t))$, $\|Y\|_p\leq C\frac{p}{q}\|Y\|_q$ for $p\geq q\geq 2$. 
Moreover, the additional condition 
$N(2t)\leq \gamma N(t)$ yields $\|Y\|_{\beta p}\geq 2 \|Y\|_p$ for $p\geq 2$ and a constant $\beta$ which depends only
on $\gamma$.
This motivates the following definitions.

\begin{dfn}
For $\alpha\geq 1$ we say that moments of a random variable $X$ \emph{grow $\alpha$-regularly} if
\[
\|X\|_p\leq \alpha\frac{p}{q}\|X\|_q\quad \mbox{ for }p\geq q\geq 2.
\]
\end{dfn}

\begin{dfn}
For $\beta<\infty$ we say that moments of a random variable $X$ \emph{grow with speed $\beta$} if
\[
\|X\|_{\beta p}\geq 2\|X\|_p \quad \mbox{ for }p\geq 2.
\]
\end{dfn}

The class of all standardized random variables with the $\alpha$-regular growth of moments
will be denoted by ${\mathcal R}_{\alpha}$ and with moments growing with speed $\beta$ by
${\mathcal S}_{\beta}$.

\begin{thm}
\label{thm:supcan}
Let $X_t=\sum_{i=1}^{\infty}t_iX_i$, $t\in \ell^2$ be the canonical process based on independent standardized
r.v.s $X_i$ with  moments growing $\alpha$-regularly with
speed  $\beta$ for some $\alpha\geq 1$ and $\beta>1$. Then for any $T\subset \ell^2$,
\[
\frac{1}{C(\alpha,\beta)}\gamma_X(T)\leq \Ex\sup_{s,t\in T}(X_s-X_t)\leq C\gamma_X(T).
\]
\end{thm}

Here and in the sequel $C(\alpha,\beta)$ denotes a constant which depends only on $\alpha$ and $\beta$ 
(which may differ at each occurrence). 
The above result easily yields the following comparison result for suprema of processes.

\begin{cor}\label{cor:compmom}
Let $X_t$ be as in Theorem \ref{thm:supcan}. Then for any nonempty  $T\subset \ell^2$ and any process $(Y_t)_{t\in T}$
such that $\|Y_s-Y_t\|_p\leq \|X_s-X_t\|_p$ for $p\geq 1$ and $s,t\in T$ we have
\[
\Ex\sup_{s,t\in T}(Y_s-Y_t)\leq C(\alpha,\beta)\Ex\sup_{s,t\in T}(X_s-X_t).
\]
\end{cor}

\begin{proof}
The assumption implies $\gamma_Y(T)\leq \gamma_X(T)$ and the result immediately follows by the lower bound in Theorem
\ref{thm:supcan} and estimate \eqref{eq:upper} used for the process $Y$.
\end{proof}

In fact one may show a stronger result.
\begin{cor}
\label{cor:comptails}
Let $X_t$ and $Y_t$ be as in Corollary \ref{cor:compmom}. Then for $u\geq 0$,
\[
\Pr\left(\sup_{s,t\in T}(Y_s-Y_t)\geq u\right)\leq 
C(\alpha,\beta)\Pr\left(\sup_{s,t\in T}(X_s-X_t)\geq \frac{1}{C(\alpha,\beta)}u\right).
\]
\end{cor}

Another consequence of Theorem \ref{thm:supcan} is the following striking bound for suprema of some canonical processes.

\begin{cor}
\label{cor:convhull}
Let $X_t$ be as in Theorem \ref{thm:supcan} and $T\subset \ell^2$ be such that $\Ex\sup_{s,t\in T}(X_s-X_t)<\infty$.
Then there exist $t^1,t^2,\ldots\in \ell^2$ such that $T-T\subset\overline{\mathrm{conv}}\{\pm t^n\colon\ n\geq 1\}$
and $\|X_{t^n}\|_{\log(n+2)}\leq C(\alpha,\beta)\Ex\sup_{s,t\in T}(X_s-X_t)$.
\end{cor}

\begin{rem}
The reverse statement easily follows by the union bound and Chebyshev's inequality. Namely, for any canonical process
$(X_t)_{t\in \ell^2}$ and any nonempty set $T\subset \ell^2$ such that  
$T-T\subset\overline{\mathrm{conv}}\{\pm t^n\colon\ n\geq 1\}$
and $\|X_{t^n}\|_{\log(n+2)}\leq M$ one has $\Ex\sup_{s,t\in T}(X_s-X_t)\leq CM$. For details see the argument after
Corollary 1.2 in \cite{BL}.
\end{rem}

\begin{rem}
Let $(\ve_i)_{i\geq 1}$ be i.i.d.\ symmetric $\pm 1$-valued r.v.s, $X_t=\sum_{i=1}^\infty t_i\ve_i$, $t\in \ell^2$ and
$T=\{e_n\colon\ n\geq 1\}$, where $(e_n)$ is the canonical basis of $\ell^2$. 
Then obviously $\Ex\sup_{s,t\in T}(X_s-X_t)=2$, moreover for any $A\subset T$
with cardinality at least 2, we have $\Delta_{2^k}(T)\geq \Delta_2(T)=\sqrt{2}$, hence $\gamma_X(T)=\infty$. Therefore
one cannot reverse bound \eqref{eq:upper} for Bernoulli processes, so some assumptions on the nontrivial speed of 
growth of moments are necessary in Theorem \ref{thm:supcan}. However, Corollary \ref{cor:convhull} holds for Bernoulli 
processes and we believe that in that statement the assumption of the $\beta$-speed of the moments growth is not needed.
\end{rem}

The crucial step in deriving lower bounds for suprema of stochastic processes is the Sudakov-minoration principle.
Following \cite{La_sud} (see also \cite{MMP}) we say that a process $(X_t)_{t\in S}$ 
satisfies \emph{the Sudakov minoration principle with constant} $\kappa>0$ if for any $p\geq 1$, 
$T\subset S$ with $|T|\geq e^p$ such that $\|X_s-X_t\|_p\geq u$ for all $s,t\in T$, $s\neq t$, 
we have $\Ex\sup_{s,t\in T}(X_s-X_t)\geq \kappa u$.

\begin{thm}
\label{thm:sud}
Suppose that $X_1,X_2,\ldots$ are independent standardized r.v.s  and moments of $X_i$ grow $\alpha$-regularly 
for some $\alpha\geq 1$. Then the canonical process 
$X_t=\sum_{i=1}^{\infty}t_iX_i$, $t\in \ell^2$ satisfies
the Sudakov minoration principle with  constant $\kappa(\alpha)$, which depends only on $\alpha$. 
\end{thm}

In fact the assumption on regular growth of moments is necessary for the Sudakov minoration principle
in the i.i.d.\ case.

\begin{prop}
\label{prop:sud_yield_reg}
Suppose that a canonical process  $X_t=\sum_{i=1}^{\infty}t_iX_i$, $t\in  \ell^2$ based on i.i.d.\  standardized
random variables $X_i$  satisfies the Sudakov minoration with  constant $\kappa>0$.
Then  moments of $X_i$ grow $C/\kappa$-regularly. 
\end{prop}

The next simple observation shows that (under mild regularity assumptions) the Sudakov minoration is necessary
for reversing bound \eqref{eq:upper}.

\begin{rem}
Suppose that for any finite $T\subset \ell^2$ we have $\Ex\sup_{s,t\in T}(X_s-X_t)\geq \kappa\gamma_X(T)$.
Assume moreover that for any $p\geq 1$ and $t\in \ell^2$, $\|X_t\|_{2p}\leq \gamma\|X_t\|_p$.
Then $X$ satisfies the Sudakov minoration principle with constant $\kappa/\gamma$.
\end{rem}

\begin{proof} Let $p\geq 1$ and $T\subset \ell^2$ of cardinality at least $e^p$ be such that 
$\|X_s-X_t\|_p\geq u$ for any $s,t\in T$, $s\neq t$.
Let $2^k\leq p< 2^{k+1}$ and $({\cal A}_n)$ be an admissible sequence of partitions of the set $T$.
Then there is $A\in {\cal A}_k$ which contains at least two points of $T$.
Hence
\[
\Ex\sup_{s,t\in T}(X_s-X_t)\geq \kappa\gamma_X(T)\geq \kappa\Delta_{2^k}(A)\geq \kappa\Delta_{\max\{p/2,1\}}(A)\geq 
\kappa u/\gamma.
\]
\end{proof}

In fact in the i.i.d.\ case we do not need the regularity assumption $\|X_t\|_{2p}\leq \gamma\|X_t\|_p$.

\begin{prop}
\label{prop:suptosud}
Let $X_t=\sum_{i=1}^{\infty}t_iX_i$, $t\in \ell^2$, where $X_i$ are i.i.d.\ standardized r.v.s.
Suppose that $\Ex\sup_{t,s\in T}X_t\geq \kappa\gamma_X(T)$ for all finite $T\subset \ell^2$. Then $(X_t)_{t\in \ell^2}$
satisfies  the Sudakov minoration principle with constant $\kappa/2$. In particular, moments of $X_i$
grow $C/\kappa$-regularly.
\end{prop}

Methods developed to prove Theorem \ref{thm:sud} also enable us to establish the following
comparison of weak and strong moments of the canonical processes based on variables with regular growth of moments.

\begin{thm}
\label{thm:weakstrong}
Let $X_t$ be as in Theorem \ref{thm:sud}. Then for any nonempty $T\subset \ell_2$ and $p\geq 1$,
\[
\left(\Ex\sup_{t\in T}|X_t|^p\right)^{1/p}\leq 
C(\alpha)\left( \Ex\sup_{t\in T}|X_t|+\sup_{t\in T}(\Ex|X_t|^p)^{1/p}\right).
\]
\end{thm}

This paper is organized as follows. In the next section we gather some general facts. In Section \ref{sec:sudakov}
we study the class ${\mathcal R}_{\alpha}$ and show that variables in this class have tails comparable to
variables with log-concave tails. Based on this observation we establish the Sudakov minoration principle 
(Theorem \ref{thm:sud}). We finish that section with  
 the proofs of Theorem \ref{thm:weakstrong} and Proposition \ref{prop:sud_yield_reg}. 
Section \ref{sec:lowerbounds} is devoted to reversing bound \eqref{eq:upper}. We obtain further regularity properties 
of the tails of variables from class ${\mathcal R}_{\alpha} \cap {\mathcal S}_{\beta}$ and then prove 
Theorem~\ref{thm:supcan} as well as 
Corollaries \ref{cor:comptails} and \ref{cor:convhull}. 
We close Section \ref{sec:lowerbounds} by  proving Proposition \ref{prop:suptosud}.

\subsection*{Notation} By $\ve_i$ we denote a Bernoulli sequence, i.e.\ a sequence of i.i.d.\ symmetric r.v.s taking values
$\pm 1$. We assume that variables $\ve_i$ are independent of other r.v.s. By a letter $C$ we denote universal
constants. Value of a constant $C$ may differ at each occurrence. Whenever we want to fix the value of an absolute
constant we use letters $C_1,C_2,\ldots$. We write $C(\alpha)$ (resp.\ $C(\alpha,\beta)$, etc.) for constants depending
only on parameters $\alpha$ (resp. $\alpha,\beta$ etc.).

\section{Preliminaries}\label{sec:prelim}

In this section we gather basic facts used in the sequel. We start with the contraction principle
for Bernoulli processes (see e.g. \cite[Theorem 4.4]{LT}).

\begin{thm}[Contraction principle]
\label{thm:contraction}
Let $(a_i)_{i=1}^n$, $(b_i)_{i=1}^n$ be two sequences of real numbers such that $|a_i| \leq |b_i|$, $i = 1, \ldots, n$.
Then
\begin{equation}
\label{eq:contraction1}
\Ex F\left(\left|\sum_{i=1}^n a_i\ve_i \right|\right) \leq \Ex F\left(\left|\sum_{i=1}^n b_i\ve_i \right|\right),
\end{equation}
where $F\colon\er_+\rightarrow\er_+$ is a convex function. In particular,
\begin{equation}
\label{eq:contraction2}
\left\|\sum_{i=1}^n a_i\ve_i\right\|_p \leq \left\|\sum_{i=1}^n b_i\ve_i\right\|_p.
\end{equation}
Moreover, for a nonempty subset $T$ of $\er^n$,
\begin{equation}
\label{eq:contraction3}
\Ex \sup_{t \in T} \sum_{i=1}^n t_ia_i\ve_i \leq \Ex \sup_{t \in T} \sum_{i=1}^n t_ib_i\ve_i.
\end{equation}
\end{thm}

The next Lemma is a standard symmetrization argument (see e.g. \cite[Lemma 6.3]{LT}).

\begin{lem}[Symmetrization]
\label{lem:symm}
Let $X_i$ be independent standardized r.v.s and $(\ve_i)$ be a Bernoulli sequence independent
of $(X_i)$. Define two canonical processes $X_t=\sum_{i=1}^\infty t_iX_i$ and its symmetrized version
$\tilde{X}_t=\sum_{i=1}^\infty t_i\ve_iX_i$. Then
\[
\frac{1}{2}\|X_s-X_t\|_p\leq \|\tilde{X}_s-\tilde{X}_t\|_p\leq 2\|X_s-X_t\|_p \quad \mbox{ for }s,t\in \ell^2 
\]
and for any $T\subset \ell^2$,
\[
\frac{1}{2}\Ex\sup_{s,t\in T}(X_s-X_t)\leq \Ex\sup_{s,t\in T}(\tilde{X}_s-\tilde{X}_t)=
2\Ex\sup_{t\in T}\tilde{X}_t\leq 2\Ex\sup_{s,t\in T}(X_s-X_t).
\]
\end{lem}

Let us also recall the Paley-Zygmund inequality (cf. \cite[Lemma 0.2.1]{KW})
which goes back to work \cite{PZ} on trigonometric series.

\begin{lem}[Paley-Zygmund inequality]
For any nonnegative random variable $S$ and $\lambda\in (0,1)$,
\begin{equation}
\label{eq:PZ}
\Pr(S\geq \lambda \Ex S)\geq (1-\lambda)^2\frac{(\Ex S)^2}{\Ex S^2}.
\end{equation}
\end{lem}

The next lemma shows that convolution preserves (up to a universal constant) the property of the $\alpha$-regular 
growth of moments. 

\begin{lem}
\label{lem:convreg}
Let $S=\sum_{i=1}^n X_i$, where $X_i$ are independent mean zero r.v.s with moments growing $\alpha$-regularly. Then
moments of $S$ grow $C\alpha$-regularly. In particular, if $(X_t)$ is a  canonical process based on
r.v.s from ${\cal R}_{\alpha}$, then $\|X_t\|_{4p}\leq C\alpha \|X_t\|_p$ for $p\geq 2$.
\end{lem}

\begin{proof}
We are to show that $\|S\|_p\leq C\alpha\frac{p}{q}\|S\|_q$ for $p\geq q\geq 2$. 
By Lemma \ref{lem:symm} we may assume that the r.v.s $X_i$ are symmetric. Moreover, by monotonicity of moments, 
it is enough to consider only the case when $p$ and $q$ are even integers and $p\geq 2q$. In \cite{La_mom} it was shown
that for $r\geq 2$,
\[
\frac{e-1}{2e^2}|||(X_i)|||_r\leq \|S\|_r\leq e|||(X_i)|||_r,
\]
where
\[
|||(X_i)|||_r:=\inf\left\{u>0\colon\ \prod_{i}\Ex\left|1+\frac{X_i}{u}\right|^r\leq e^r\right\}.
\]
Therefore it is enough to proof the following claim.

\medskip

\noindent
{\bf Claim.} Suppose that $Y$ is a symmetric r.v.\ with moments growing $\alpha$-regularly. Let $p,q$ be positive even
integers such that $p\geq 2q$ and $\Ex|1+Y|^q\leq e^A\leq e^q$. Then $\Ex|1+\frac{q}{4e\alpha p}Y|^p\leq e^{pA/q}$.
\medskip

To show the claim first notice that
\[
\Ex|1+Y|^q=1+\sum_{k=1}^{q/2}\binom{q}{2k}\Ex|Y|^{2k}\geq 1+\sum_{k=1}^{q/2}\left(\frac{q}{2k}\right)^{2k}\Ex|Y|^{2k}
\geq 1+\Ex|Y|^q.
\] 
In particular, $\|Y\|_q\leq (e^{A}-1)^{1/q}\leq e$.
On the other hand,
\[
\Ex\left|1+\frac{q}{4e\alpha p}Y\right|^p=1+\sum_{k=1}^{p/2}\binom{p}{2k}\Ex\left|\frac{q}{4e\alpha p}Y\right|^{2k}\leq
1+\sum_{k=1}^{p/2}\left(\frac{q}{8\alpha k}\right)^{2k}\Ex|Y|^{2k}.
\]
Since $\alpha\geq 1$ we obviously have
\[
1+\sum_{k=1}^{q/2}\left(\frac{q}{8\alpha k}\right)^{2k}\Ex|Y|^{2k}\leq \Ex|1+Y|^q\leq e^A.
\]
The $\alpha$-regularity of moments of $Y$ yields
\[
\sum_{k=q/2+1}^{p/2}\left(\frac{q}{8\alpha k}\right)^{2k}\Ex|Y|^{2k}\leq
\sum_{k=q/2+1}^{p/2} \left(\frac{1}{4}\|Y\|_q\right)^{2k}\leq
\left(\frac{1}{4}\|Y\|_q\right)^{q}\sum_{l=1}^{\infty}\left(\frac{e}{4}\right)^{2l}\leq \|Y\|_q^q.
\]
Thus 
\[
\Ex\left|1+\frac{q}{4e\alpha p}Y\right|^p\leq e^A+\|Y\|_q^q\leq 2e^{A}-1\leq e^{2A}\leq e^{pA/q},
\]
which completes the proof of the claim and of the lemma. 
\end{proof}

We finish this section with the observation that will allow us to compare regular r.v.s with variables 
with log-concave tails.

\begin{lem}
\label{lm:convex}
Let a nondecreasing function $f\colon\er_+\rightarrow\er_+$ satisfy
\[
f(c\lambda t) \geq \lambda f(t), \qquad \mbox{for } \lambda \geq 1,\ t \geq t_0,
\]
where $t_0 \geq 0$, $c \geq 2$ are some constants. Then there is a function 
$g\colon\er_+\rightarrow\er_+$, convex on $[ct_0, \infty)$, such that
\[
g(t) \leq f(t) \leq g(c^2t), \qquad \mbox{for }t \geq ct_0,
\]
and $g(ct_0) = 0$.
\end{lem}

\begin{proof}
For $t \geq ct_0$ we set
\[
g(t) := \int_{ct_0}^t \sup_{ct_0 \leq y \leq x} \frac{f(y/c)}{y} \dd x.
\]
Then $g$ is convex on $[ct_0, \infty)$ as an integral of a nondecreasing function. For $t \geq x\geq ct_0$ we have 
$\sup_{ct_0 \leq y \leq x} f(y/c)/y \leq f(t)/t$, as $f(\lambda y)/(\lambda y) \geq f(y/c)/y$ for $y \geq ct_0$ and
$\lambda\geq 1$. Thus
\[
g(t) \leq (t-ct_0)\frac{f(t)}{t} \leq f(t), \qquad \mbox{for } t \geq ct_0.
\]
Moreover, for $t \geq ct_0$
\begin{align*}
g(ct) 
&= \int_{ct_0}^{ct} \sup_{ct_0 \leq y \leq x} \frac{f(y/c)}{y} \dd x 
\geq \int_t^{ct} \sup_{ct_0 \leq y \leq x} \frac{f(y/c)}{y} \dd x 
\\
&\geq (ct-t)\frac{f(t/c)}{t} = (c-1)f(t/c) \geq f(t/c),
\end{align*}
hence $g(c^2t) \geq f(t)$ for $t \geq ct_0$.
\end{proof}

\section{Sudakov minoration principle}\label{sec:sudakov}

The main goal of this section is to prove Theorem \ref{thm:sud}. The strategy of the proof is to reduce the problem involving 
random variables with moments growing regularly to the case of random variables with log-concave tails, for which the 
minoration is known 
(see \cite[Theorem 1]{La_sup}). The relevant result can be restated as follows

\begin{thm}
\label{thm:latala}
Let $X_t=\sum_{i=1}^\infty t_iX_i$, $t\in \ell^2$ be the canonical process based on independent symmetric r.v.s $X_i$
with log-concave tails. Then $(X_t)_{t\in \ell^2}$ satisfies the Sudakov minoration principle with 
a universal constant $\kappa_{\mathrm{lct}}>0$.
\end{thm}

\begin{rem}
Since we may normalize $X_i$ we  do not need to assume that they have variance one. It suffices to have
$\sup_{i}\Var(X_i)<\infty$ in order that $X_t$ is well defined for $t\in \ell^2$.
\end{rem}

The mentioned reduction hinges on the idea that the tail functions of random variables with regular growth of moments 
ought to be close to log-concave functions as, conversely, log-concave random variables are regular.

\begin{prop}
\label{prop:betweenconvex}
Let $\alpha \geq 1$. There exist constants $T_\alpha, L_\alpha$ such that for any $X \in \mathcal{R}_\alpha$ 
there is a nondecreasing function $M\colon [0,\infty)\rightarrow [0,\infty]$ which is convex, $M(T_\alpha) = 0$, 
and satisfies
\begin{equation}
\label{eq:betweenconvex}
M(t) \leq N(t) \leq M(L_\alpha t), \qquad \mbox{for }t \geq T_\alpha,
\end{equation}
where $N(t) = -\ln\Pr(|X| > t)$.
\end{prop}

\begin{proof}
Fix $\alpha \geq 1$. We begin with showing that there is a constant $\kappa_\alpha$ such that for any 
$X \in \mathcal{R}_\alpha$,
\begin{equation}
\label{eq:sublinear}
N(\kappa_\alpha \lambda t) \geq \lambda N(t), \qquad \lambda \geq 1,\ t \geq 1 - 1/e.
\end{equation}
When $\|X\|_\infty < \infty$ it is enough to prove this assertion for $t < (1-1/e)\|X\|_\infty$ as, providing that 
$\kappa_\alpha \geq (1-1/e)^{-1}$, for $t \geq (1-1/e)\|X\|_\infty$ we have 
$N(\kappa_\alpha \lambda t) \geq N\left( \n{X}_\infty \right) = \infty$.

So, fix $\lambda \geq 1$ and $1-1/e\leq t < (1-1/e)\|X\|_\infty$. There exists $q \geq 2$ such that $t = (1-1/e)\|X\|_q$. 
Pick also $p \geq q$ so that $\lambda = p/q$. By the Paley-Zygmund inequality \eqref{eq:PZ} and
by the assumption that $X \in \mathcal{R}_\alpha$ we obtain
\begin{align}
\notag
N(t) 
&= N\left( (1-1/e)\|X\|_q \right) \leq N\left( (1-1/e)^{1/q}\|X\|_q \right) 
\\
\notag
&= -\ln \Pr(|X|^q > (1-1/e)\Ex|X|^q) \leq -\ln \left( \frac{1}{e^2}\left( \frac{\|X\|_q}{\|X\|_{2q}} \right)^{2q} \right) 
\\
\label{eq:bc1}
&\leq 2 + q\ln\left[(2\alpha)^2\right] \leq q\ln\left(e(2\alpha)^2\right) =: qb_\alpha.
\end{align}
On the other hand, setting $\kappa_\alpha = e^{b_\alpha}(1-1/e)^{-1}\alpha$, with the aid of the assumption that 
$X \in \mathcal{R}_\alpha$ and Chebyshev's inequality, we get
\begin{align}
\notag
N(\kappa_\alpha \lambda t) 
&= N\left( e^{b_\alpha} \alpha\frac{p}{q}\|X\|_q \right) \geq N\left( e^{b_\alpha} \|X\|_p \right) 
\\
\label{eq:bc2}
&= -\ln \Pr(|X|^p > e^{pb_\alpha} \Ex|X|^p) \geq pb_\alpha = \lambda qb_\alpha.
\end{align}
Joining inequalities \eqref{eq:bc1} and \eqref{eq:bc2} we get \eqref{eq:sublinear} with $\kappa_\alpha = \frac{4e^2}{e-1}\alpha^3$.

By virtue of this \emph{sublinear} property \eqref{eq:sublinear}, Lemma \ref{lm:convex} applied to $f = N$, 
$c = \kappa_\alpha$, and $t_0 = 1-1/e$ finishes the proof, providing the constants
\[
L_\alpha = \kappa_\alpha^2 = \left(\frac{4e^2}{e-1}\alpha^3\right)^2, \qquad T_\alpha = \kappa_\alpha t_0 = 4e\alpha^3.
\]
\end{proof}

\begin{proof}[Proof of Theorem \ref{thm:sud}]
We fix $p\geq 2$, $T\subset \ell^2$ such that $|T|\geq e^p$ and $\|X_s-X_t\|_p\geq u$ for all distinct $s,t\in T$.
We are to show that $\Ex\sup_{s,t\in T}(X_s-X_t)\geq \kappa_{\alpha}u$
for a constant $\kappa_{\alpha}$ which depends only on $\alpha$. 
By Lemma \ref{lem:symm} we may assume that r.v.s $X_i$ are symmetric.

Proposition \ref{prop:betweenconvex} yields that the tail functions $N_i(t) := -\ln \Pr(|X_i| > t)$ 
of the variables $X_i$ are controlled by the convex functions $M_i(t)$, apart from $t \leq T_\alpha$, i.e. we have
$M_i(t) \leq N_i(t) \leq M_i(L_\alpha t)$ only for $t\geq T_{\alpha}$.
To gain control also for $t \leq T_{\alpha}$, define the symmetric random variables
\[
\wt{X}_i = (\sgn X_i)\max\{|X_i|,T_\alpha\},
\]
so that their tail functions $\wt{N}_i(t) = -\ln\Pr(|\wt{X}_i|>t)$,
\[
\wt{N_i}(t)=
\begin{cases}
0, & t < T_{\alpha} 
\\
N_i(t), & t \geq T_{\alpha}
\end{cases},
\]
satisfy
\begin{equation}
\label{eq:compofNandM}
M_i(t) \leq \wt{N}_i(t) \leq M_i(L_\alpha t) \qquad \mbox{for all }t \geq 0.
\end{equation}
This allows us to construct a sequence $Y_1, Y_2, \ldots$ of independent symmetric r.v.s with log-concave tails 
given by $\Pr(|Y_i| > t) = e^{-M_i(t)}$ such that
\begin{equation}
\label{eq:YXtildeequiv}
|Y_i| \geq |\wt{X}_i| \geq \frac{1}{L_\alpha}|Y_i|.
\end{equation}
Define the canonical processes $\wt{X}_t:=\sum_{i=1}^{\infty}t_i\wt{X}_i$ and $Y_t:=\sum_{i=1}^{\infty}t_iY_i$, $t\in \ell^2$. 

Since $|Y_i|\geq |X_i|$ and variables $Y_i$ and $X_i$ are symmetric we get for $s,t\in T$, $s\neq t$,
\[
\|Y_s-Y_t\|_p= \left\|\sum_{i=1}^\infty (s_i-t_i)|Y_i|\ve_i\right\|_p
\geq \left\|\sum_{i=1}^\infty (s_i-t_i)|X_i|\ve_i\right\|_p=\|X_s-X_t\|_p\geq u,
\]
where the first inequality follows by contraction principle \eqref{eq:contraction2} as $|Y_i| \geq |\wt{X}_i| \geq |X_i|$.
Hence we can apply Theorem \ref{thm:latala} to the canonical process $(Y_t)$ and 
obtain
\begin{equation}
\label{eq:sudY}
2\Ex\sup_{t\in T}Y_t=\Ex \sup_{s,t \in T}(Y_s-Y_t) \geq \kappa_{\mathrm{lct}}u.
\end{equation}

To finish the proof it suffices to show that $\E\sup_{t \in T} X_t$ majorizes $\E\sup_{t\in T} Y_t$.
Clearly,
\begin{equation}\label{eq:obvtriangle}
\Ex\sup_{t \in T} X_t \geq \Ex\sup_{t \in T} \wt{X}_t - \Ex\sup_{t \in T} (\wt{X}_t-X_t).
\end{equation}
Recall that by the definition of $\wt{X}_i$, 
$|\wt{X}_i - X_i| = |T_\alpha - X_i|\1_{\{|X_i|\leq T_\alpha\}} \leq T_\alpha$.
As a consequence, the supremum of the canonical process $\E\sup_{t \in T} (\wt{X}_t-X_t)$ is bounded by the supremum of 
the Bernoulli process $\E\sup_{t \in T} \sum t_iT_{\alpha}\ve_i$. Indeed, using the symmetry of the distribution of 
the variables $\wt{X}_i - X_i$ and contraction principle \eqref{eq:contraction3},
\[
\Ex\sup_{t \in T} (\wt{X}_t-X_t) = \Ex_X\Ex_\ve\sup_{t \in T} \sum_{i=1}^{\infty} t_i|\wt{X}_i-X_i|\ve_i 
\leq \Ex_\ve \sup_{t \in T} \sum_{i=1}^{\infty} t_iT_\alpha\ve_i.
\]

Since $X_i\in {\cal R}_{\alpha}$ we get by H\"older's inequality,
\begin{align*}
1 = \Ex X_i^2 = \Ex X_i^{4/3}X_i^{2/3} \leq \|X_i\|_4^{4/3}\|X_i\|_1^{2/3} \leq (2\alpha\|X_i\|_2)^{4/3}\|X_i\|_1^{2/3}=(2\alpha)^{4/3}(\Ex|X_i|)^{2/3}
\end{align*}
and thus $\Ex|X_i|\geq (2\alpha)^{-2}$. Hence by Jensen's inequality
\[
\Ex\sup_{t \in T} X_t = \Ex_\ve\Ex_X \sup_{t\in T} \sum_{i=1}^{\infty} t_i|X_i|\ve_i 
\geq \Ex_\ve \sup_{t \in T} \sum_{i=1}^{\infty} t_i\Ex_X|X_i|\ve_i \geq 
\frac{1}{(2\alpha)^2} \Ex \sup_{t \in T} \sum_{i=1}^{\infty} t_i\ve_i.
\]
As a result,
\[
\Ex\sup_{t \in T} (\wt{X}_t-X_t) \leq (2\alpha)^2 T_\alpha \E\sup_{t \in T} X_t,
\]
and by \eqref{eq:obvtriangle},
\begin{equation}
\label{eq:supXtildeX}
\Ex\sup_{t \in T} X_t \geq \frac{1}{1+(2\alpha)^2 T_\alpha} \Ex\sup_{t \in T} \wt{X}_t.
\end{equation}
Finally, notice that, by virtue of contraction principle \eqref{eq:contraction3}, the second inequality of 
\eqref{eq:YXtildeequiv} implies that
\begin{equation}
\label{eq:suptileXY}
\Ex\sup_{t \in T} \wt{X}_t \geq \frac{1}{L_\alpha} \Ex\sup_{t \in T} Y_t.
\end{equation}

Estimates \eqref{eq:sudY}, \eqref{eq:supXtildeX} and \eqref{eq:suptileXY} yield
\[
\Ex\sup_{s,t \in T}(X_s-X_t)=2\Ex\sup_{t \in T} X_t\geq \frac{2}{L_{\alpha}(1+(2\alpha)^2 T_\alpha)}\Ex\sup_{t \in T} Y_t
\geq \frac{\kappa_{\mathrm{lct}}}{L_{\alpha}(1+(2\alpha)^2 T_\alpha)}u.
\]
\end{proof}

\begin{proof}[Proof of Theorem \ref{thm:weakstrong}]
Using a symmetrization argument we may assume that the variables $X_i$ are symmetric. Let variables
$\wt{X}_i, Y_i$ and the related canonical processes be as in the proof of Theorem \ref{thm:sud}. 
Since the variables $Y_i$ have log-concave tails by \cite{La_SM} we get
\[
\left(\Ex\sup_{t\in T}|Y_t|^p\right)^{1/p}\leq 
C\left( \Ex\sup_{t\in T}|Y_t|+\sup_{t\in T}(\Ex|Y_t|^p)^{1/p}\right).
\]
Estimate $|Y_i|\geq |X_i|$ and the contraction principle yield
\[
\Ex\sup_{t\in T}|X_t|^p\leq \Ex\sup_{t\in T}|Y_t|^p.
\]
We showed above that 
\[
\Ex\sup_{t\in T}|Y_t|\leq L_\alpha(1+(2\alpha)^2T_\alpha)\Ex\sup_{t\in T}|X_t|.
\]
Finally the contration principle together with the bounds $|Y_i|\leq L_{\alpha}|\wt{X}_i|$, $|X_i-\wt{X}_i|\leq T_{\alpha}$
and $\Ex |X_i|\geq (2\alpha)^{-2}$ imply 
\[
\|Y_t\|_p\leq L_{\alpha}\|\wt{X}_t\|_p\leq L_{\alpha}\|X_t\|_p+L_{\alpha}T_{\alpha}\left\|\sum_{i=1}^{\infty}t_i\ve_i\right\|_p\leq
L_\alpha(1+T_{\alpha}(2\alpha)^2)\|X_t\|_p.
\]
\end{proof}

We conclude this section  with the proof of Proposition \ref{prop:sud_yield_reg} showing that in the i.i.d.\ case
the Sudakov minoration principle and the $\alpha$-regular growth of moments are  equivalent.

\begin{proof}[Proof of Proposition \ref{prop:sud_yield_reg}]
Let us fix $p\geq q\geq 2$ and for $1\leq m\leq n$ consider the following subset of $\ell^2$ 
\[
T=T(m,n) = \left\{t \in \ell^2\colon\ \sum_{i=1}^n t_i = m, \ t_i = 0, i > n\right\}.
\] 
Then $|T|=\binom{n}{m}\geq (n/m)^m\geq e^p$ if $n\geq me^{p/m}$. 
Moreover, for any $s, t \in T$, $s \neq t$, say with $s_j \neq t_j$ we have $\|X_s-X_t\|_p \geq \|X_j\|_p$.
Thus the Sudakov minoration principle yields for any $n\geq me^{p/m}$,
\begin{equation}
\label{eq:sud_yields_reg1}
\kappa\|X_i\|_p 
\leq \Ex\sup_{s,t \in T} (X_s-X_t) \leq 2\Ex\sup_{\substack{I \subset [n] \\ |I| = m}} \sum_{i \in I} |X_i| 
=2\Ex\sum_{k=1}^m X_k^*,
\end{equation}
where $(X_1^*,X_2^*,\ldots,X_n^*)$ is the nonincreasing rearrangement of $(|X_1|,|X_2|,\ldots,|X_n|)$.

We have
\begin{align*}
\Pr(X_k^*\geq t)=\Pr\left(\sum_{i=1}^n\1_{\{|X_i|\geq t\}}\geq k\right)\leq \frac{1}{k}\sum_{i=1}^n\Ex \1_{\{|X_i|\geq t\}}
=\frac{n}{k}\Pr(|X_i|\geq t)\leq \frac{n}{k}\frac{\|X_i\|_q^q}{t^q}.
\end{align*}
Integration by parts shows that 
\[
\Ex X_k^*\leq C\left(\frac{n}{k}\right)^{1/q}\|X_i\|_q.
\]
Combining this with \eqref{eq:sud_yields_reg1} we get (recall that $q\geq 2$ and constant $C$ may differ at each
occurrence)
\[
\kappa\|X_i\|_p\leq C\sum_{k=1}^m\left(\frac{n}{k}\right)^{1/q}\|X_i\|_q\leq Cn^{1/q}m^{1-1/q}\|X_i\|_q.
\]

Taking $m=\lceil p/q\rceil$ and $n=\lceil me^{p/m}\rceil$ we find that 
$n^{1/q}m^{1-1/q} \leq 4ep/q$. Hence
\[
\|X_i\|_p \leq \frac{C}{\kappa}\frac{p}{q}\|X_i\|_q
\]
which finishes the proof.
\end{proof}

\section{Lower bounds for suprema of canonical processes}
\label{sec:lowerbounds}

As in the case of the Sudakov minoration principle the proof of the lower bound in Theorem \ref{thm:supcan} is based
on the corresponding result for the canonical processes built on variables with log-concave tails.
Theorem 3 in \cite{La_sup} (see also Theorem 10.2.7 and Exercise 10.2.14  in \cite{Talnew}) implies the following result.

\begin{thm}
\label{thm:suplct}
Let $X_t=\sum_{i=1}^\infty t_iX_i$, $t\in \ell^2$ be the canonical process based on independent symmetric r.v.s $X_i$
with log-concave tails. Assume moreover that there exists $\gamma$ such that $N_i(2t)\leq \gamma N_i(t)$ for all 
$i$ and $t>0$, where $N_i(t)=-\ln\Pr(|X_i|> t)$. Then there exists a constant $C_{\mathrm{lct}}(\gamma)$, which 
depends only on $\gamma$ such that for any $T\subset \ell^2$,
\[
\Ex\sup_{s,t\in T}(X_s-X_t)=2\Ex\sup_{t\in T}X_t\geq \frac{1}{C_{\mathrm{lct}}(\gamma)}\gamma_X(T).
\]
\end{thm}

\begin{rem}
Theorem 3 in \cite{La_sup} and Theorem 10.2.7 in \cite{Talnew}  were formulated in a slightly different language. 
In particular, the latter states that there exist $r>2$, an admissible sequence of partitions $({\cal A}_{n})$
and numbers $j_n(A)$ for $A\in {\cal A}_n$ such that $\varphi_{j_n(A)}(s,s')\leq 2^{n+1}$ for all
$s,s'\in A$  and 
\[
\sup_{t\in T}\sum_{n=0}^{\infty}2^{n}r^{-j_n(A_n(t))}\leq C(\gamma)\Ex\sup_{t\in T} X_t.
\]
 (For the definition of $\varphi$ see \cite{Talnew} - it precedes the statement of Theorem 10.2.7.) However, the condition $\varphi_{j_n(A)}(s,s')\leq 2^{n+1}$ yields that $\|X_s-X_{s'}\|_{2^{n}}\leq C2^nr^{-j_n(A)}$
(see \cite{GK} for the i.i.d. case and Example 3 in \cite{La_mom} for the general situation),
so $\Delta_{2^{n}}(A_n(t))\leq C2^nr^{-j_n(A_n(t))}$ and
\[
\gamma_X(T)\leq C\sup_{t\in T}\sum_{n=0}^{\infty}2^{n}r^{-j_n(A_n(t))}\leq C_{\mathrm{lct}}(\gamma)\Ex\sup_{t\in T} X_t.
\]
\end{rem}

\begin{prop}
\label{prop:taillowgrowth}
Let $\alpha\geq 1, \beta > 1$. For any $r > 1$ there exists a constant $C(\alpha,\beta,r)$ such that 
for  $X \in \mathcal{R}_\alpha \cap \mathcal{S}_\beta$ we have
\begin{equation}
N(rt) \leq C(\alpha,\beta,r) N(t), \qquad t \geq 2,
\end{equation}
where $N(t) := -\ln \Pr(|X| > t)$.
\end{prop}

\begin{proof}
Fix $t \geq 2$ and define 
\[
q:=\inf\{p\geq 2\colon\ \|X\|_{\beta p}\geq t\}.
\]
Since $X \in \mathcal{R}_\alpha \cap \mathcal{S}_\beta$, the function $p\longmapsto \|X\|_p$ is finite
and continuous on $[2,\infty)$, moreover $\|X\|_2=1$ and $\|X\|_{\infty}=\infty$. Hence, if
$t\geq \|X\|_{2\beta}$, we have $t=\|X\|_{\beta q}$ and  by Chebyshev's inequality,
\[
N(t) = N( \|X\|_{\beta q}) \geq N(2 \|X\|_q) = -\ln \Pr(|X|^q > 2^q\E|X|^q) 
\geq q\ln 2.
\]
If $2\leq t<\|X\|_{2\beta}$, then $q=2$ and
\[
N(t)\geq N(2)= -\ln \Pr(|X|^2 > 4\Ex|X|^2)  \geq \ln 4=q\ln 2.
\]

Set an integer $k$ such that $r \leq 2^{k-2}$. Then, using consecutively the definition of $q$, the assumption that 
$X \in \mathcal{S}_\beta$, the Paley-Zygmund inequality, and the assumption that $X \in \mathcal{R}_\alpha$, 
we get the estimates
\begin{align}
\notag
N(rt) 
&\leq N\left( 2^{k-2}\|X\|_{\beta q} \right) \leq N\left(\frac{1}{2}\|X\|_{\beta^kq} \right) 
= -\ln \Pr\left( |X|^{\beta^kq} > 2^{-\beta^k q}\Ex|X|^{\beta^kq} \right) 
\\
\label{eq:tailgr}
&\leq -\ln \left( \frac{1}{4} \left( \frac{\|X\|_{\beta^kq}}{\|X\|_{2\beta^{k}q}} \right)^{2\beta^kq} \right)  
\leq \ln 4 + 2\beta^kq\ln (2\alpha)\leq q(\ln 2+2\beta^k\ln(2\alpha)).
\end{align}

Combining the above estimates we obtain the assertion with 
$C(\alpha,\beta,r) =(\ln 2+2\beta^k\ln(2\alpha))/\ln 2$ and 
 $k = k(r)$ being an integer such that $2^{k-2} \geq r$.
\end{proof}

\begin{remark}
\label{rem:Nbound}
Taking in \eqref{eq:tailgr} $t = 2$ which corresponds to $q = 2$ we find that 
\[
N(s) \leq 2(\ln 2+2\beta^k\ln(2\alpha)),\qquad \mbox{for }s<2^{k-1},
\]
which means that the tail distribution function of a variable $X \in \mathcal{R}_\alpha\cap \mathcal{S}_\beta$ 
at a certain value $s$ is bounded with a constant not depending on the distribution of $X$ but only on the 
parameters $\alpha, \beta$ and of course the value of $s$.
\end{remark}

\begin{proof}[Proof of Theorem \ref{thm:supcan}]
In view of \eqref{eq:upper} we are to address only the lower bound on $\Ex \sup_{t \in T} X_t$. A symmetrization argument
shows that we may assume that variables $X_i$ are symmetric.

Given symmetric $X_i$ let $Y_i$ be random variables defined as in the proof  of Theorem \ref{thm:sud}, 
i.e.\ $Y_i$'s are independent symmetric r.v.s having log-concave tails $\Pr(|Y_i| > t) = e^{-M_i(t)}$. 
Moreover, let $L_\alpha, T_\alpha$ be the constants as in Proposition \ref{prop:betweenconvex}.  
Due to Proposition \ref{prop:taillowgrowth} for $r = 2L_\alpha$ we know that the functions $N_i(t) := -\Pr(|X_i| > t)$ 
satisfy
\[
N_i(2L_\alpha t) \leq \gamma N(t), \qquad t \geq 2,
\]
where $\gamma=\gamma(\alpha,\beta):=C(\alpha,\beta,2L_{\alpha})$.

What then can be said about $M_i$? Using \eqref{eq:betweenconvex} we find that for 
$t \geq \tilde{T}_\alpha:=\max\{2,T_{\alpha}\}$
\[
M_i(2L_\alpha t) \leq N_i(2L_\alpha t) \leq \gamma N_i(t) \leq \gamma M_i(L_\alpha t),
\]
which means that $M_i$ are almost of moderate growth, namely for $t_{\alpha}:=L_\alpha\tilde{T}_\alpha$ we have 
\[
M_i(2t) \leq \gamma M_i(t), \qquad t \geq t_\alpha.
\]
Therefore, we improve the function $M_i$ putting on the interval $[0, t_\alpha]$ an artificial linear piece 
$t \mapsto \lambda(i,\alpha) t$, where $\lambda(i,\alpha):= M_i(t_\alpha)/t_\alpha$. In other words, take the numbers 
$p(i,\alpha) := \Pr(|Y_i| > t_\alpha) = e^{-M_i(t_\alpha)}$ and let $U_i$ be a sequence of independent random variables 
with the following symmetric \emph{truncated} exponential distribution,
\[
\Pr(|U_i| > t) = 
\begin{cases}
\frac{e^{-\lambda(i,\alpha) t}-p(i,\alpha)}{1-p(i,\alpha)}, & t \leq t_\alpha 
\\
0, & t > t_\alpha
\end{cases},
\]
which are in addition independent of the sequences $(X_i)$ and $(Y_i)$. Define
\[
Z_i := Y_i\1_{\{|Y_i| > t_\alpha\}} + U_i\1_{\{|Y_i| \leq t_\alpha\}}. 
\]
Let
\[
\wt{M}_i(t): = -\ln \Pr(|Z_i| > t) = 
\begin{cases}
\lambda(i,\alpha) t, & t \leq t_\alpha, 
\\
M_i(t), & t > t_\alpha.
\end{cases}
\]
Then $\wt{M}_i$ are convex functions of moderate growth, i.e.
\[
\wt{M}_i(2t) \leq \tilde{\gamma} \wt{M}_i(t), \qquad t \geq 0,
\]
where $\tilde{\gamma}=\tilde{\gamma}(\alpha,\beta):=\max\{2,\gamma\}$.

Thus Theorem \ref{thm:suplct} can be applied to the canonical process $Z_t:=\sum_{i}t_iZ_i$ and we get
\[
\Ex \sup_{t\in T}Z_t\geq \frac{1}{C_1(\alpha,\beta)}\gamma_Z(T),
\]
where $C_1(\alpha,\beta)=C_{\mathrm{lct}}(\tilde{\gamma})$.

What is left is to compare both the suprema and the functionals $\gamma$'s of the processes $(X_t)$ and $(Z_t)$. 
The former is easy, because we have $M_i(t) \leq \wt{M}_i(t)$, $t \geq 0$, which allows to take samples such  
that $|Y_i| \geq |Z_i|$, and consequently, thanks to contraction principle \eqref{eq:contraction3}, 
$\Ex\sup_{t \in T} Z_t \leq \Ex\sup_{t \in T} Y_t$. Joining this with estimates \eqref{eq:suptileXY} and 
\eqref{eq:supXtildeX} we derive
\[
\Ex \sup_{t \in T} Z_t \leq L_\alpha(1 + (2\alpha)^2 T_\alpha) \E \sup_{t \in T} X_t.
\]

For the latter, we would like to show $C(\alpha,\beta)\gamma_Z \geq \gamma_X$ . It is enough to compare the metrics,
 i.e.\ to prove that $C(\alpha,\beta)\|Z_s-Z_t\|_p \geq \|X_s-X_t\|_p$ for $p\geq 1$. We proceed as in the proof of 
 Theorem \ref{thm:sud}. We have
\begin{equation}
\label{eq:momentsZtriang}
\|Z_s-Z_t\|_p \geq \|Y_s-Y_t\|_p - \|(Y_s-Z_s)-(Y_t-Z_t)\|_p.
\end{equation}
In the proof of Theorem \ref{thm:sud} it was established that $\|Y_s-Y_t\|_p \geq \|X_s-X_t\|_p$. 
For the second term we use the symmetry of the variables $Y_i-Z_i$, contraction principle \eqref{eq:contraction2}, 
and the fact that $|Y_i - Z_i| \leq 2t_\alpha$, obtaining
\begin{equation}
\label{eq:momentsY-Z}
 \|(Y_s-Z_s)-(Y_t-Z_t)\|_p = \left\|\sum_{i}(s_i-t_i)|Y_i - Z_i|\ve_i\right\|_p 
 \leq 2t_\alpha\left\|\sum_i (s_i-t_i)\ve_i\right\|_p.
\end{equation}
Now we compare $\|Z_s-Z_t\|_p$ with moments of increments of the Bernoulli process. By Jensen's inequality we get
\begin{equation}
\label{eq:momentsZ}
\|Z_s-Z_t\|_p = \left\|\sum_i (s_i-t_i)|Z_i|\ve_i\right\|_p \geq 
\min_i \Ex|Z_i| \left\|\sum_i (s_i-t_i)\ve_i\right\|_p.
\end{equation}
Combining \eqref{eq:momentsZtriang}, \eqref{eq:momentsY-Z}, and \eqref{eq:momentsZ} yields
\[
\|Z_s-Z_t\|_p \geq \left( 1 + \frac{2t_\alpha}{\min_i\Ex |Z_i|} \right)^{-1}\|X_s-X_t\|_p.
\]
To finish it suffices to prove that $\Ex|Z_i| \geq c_{\alpha,\beta}$ for some positive constant $c_{\alpha,\beta}$, which
depends only on $\alpha$ and $\beta$. This is a cumbersome yet simple calculation. 
Recall the distributions of the variables $Y_i$ and $U_i$, 
the fact that they are independent, and observe that
\begin{align*}
\Ex|Z_i| 
&= \E|Y_i|\1_{\{|Y_i| > t_\alpha\}} + \Ex|U_i|\1_{\{|Y_i| \leq t_\alpha\}} 
\\ 
&\geq t_\alpha \Pr(|Y_i| > t_\alpha) + \left( \Ex|U_i| \right)\Pr(|Y_i| \leq t_\alpha)
\\
&= t_\alpha p(i,\alpha) + (1 - p(i,\alpha))\int_0^{t_\alpha} \frac{e^{-\lambda(i,\alpha) t}-p(i,\alpha)}{1-p(i,\alpha)} \dd t
\\
&=\frac{1}{\lambda(i,\alpha)}\left( 1 - e^{-\lambda(i,\alpha) t_\alpha} \right) = 
\frac{t_\alpha}{M_i(t_\alpha)}\left( 1 - e^{-M_i(t_\alpha)} \right).
\end{align*}
The last expression is nonincreasing with respect to $M_i(t_\alpha)$.
Since $M_i(t_\alpha) \leq N_i(t_\alpha)$ (see \eqref{eq:betweenconvex}), we are done provided that we can bound $N_i(t_\alpha)$ above.
Thus, Remark \ref{rem:Nbound} completes the proof.
\end{proof}

\begin{proof}[Proof of Corollary \ref{cor:comptails}]
Proposition 20 in \cite{La_sud} yields for $p\geq 1$,
\begin{align*}
\left(\Ex\sup_{t,s\in T}|Y_t-Y_s|^p\right)^{1/p}&\leq C(\gamma_Y(T)+\sup_{s,t\in T}\|Y_s-Y_t\|_p)\leq
C(\gamma_X(T)+\sup_{s,t\in T}\|X_s-X_t\|_p)
\\
& \leq C(\alpha,\beta) \left(\Ex\sup_{s,t\in T}|X_s-X_t|+\sup_{s,t\in T}\|X_s-X_t\|_p\right)
\\
&\leq
(C(\alpha,\beta)+1)\left\|\sup_{s,t\in T}|X_s-X_t|\right\|_{p},
\end{align*}
where the third inequality follows by Theorem \ref{thm:supcan}.
Hence by Chebyshev's inequality we obtain 
\begin{equation}
\label{eq:estY}
\Pr\left(\sup_{t,s\in T}|Y_t-Y_s|\geq C_1(\alpha,\beta)\left\|\sup_{s,t\in T}|X_s-X_t|\right\|_{p}\right)
\leq e^{-p}
\quad \mbox{ for }p\geq 1.
\end{equation}

Theorem \ref{thm:weakstrong} (used for the set $T-T$) and Lemma \ref{lem:convreg} yield for $p\geq q\geq 1$,
\[
\left\|\sup_{s,t\in T}|X_s-X_t|\right\|_{p}\leq C_2(\alpha)\frac{p}{q}\left\|\sup_{s,t\in T}|X_s-X_t|\right\|_{q}.
\]
Hence, by the Paley-Zygmund inequality we get for $q\geq 1$,
\[
\Pr\left(\sup_{t,s\in T}|X_t-X_s|\geq \frac{1}{2}\left\|\sup_{s,t\in T}|X_s-X_t|\right\|_{q}\right)
\geq \frac{1}{4}\left(\frac{1}{2C_2(\alpha)}\right)^q.
\]
Applying the above estimate with $q=p/\ln(2C_2(\alpha))$ we get
\begin{equation}
\label{eq:estX}
\Pr\left(\sup_{t,s\in T}|X_t-X_s|\geq \frac{1}{2C_2(\alpha)\ln(2C_2\alpha)}
\left\|\sup_{s,t\in T}|X_s-X_t|\right\|_{p}\right)
\geq \frac{1}{4}e^{-p} \quad \mbox{ for }p\geq \ln(2C_2(\alpha)).
\end{equation}

The assertion easily follows by \eqref{eq:estY} and \eqref{eq:estX}.

\end{proof}

\begin{proof}[Proof of Corollary \ref{cor:convhull}]
By Theorem \ref{thm:supcan} we may find an admissible sequence of partitions $({\cal A}_n)$ such that
\begin{equation}
\label{eq:part}
\sup_{t\in T}\sum_{n=0}^\infty \Delta_{2^n}(A_n(t))\leq C(\alpha,\beta)\Ex\sup_{s,t\in T}(X_s-X_t).
\end{equation}
For any $A\in {\cal A}_n$ let us choose a point $\pi_n(A)\in A$ and set $\pi_n(t):=\pi_n(A_n(t))$. 

Let $M_n:=\sum_{j=0}^{n} N_j$ for $n=0,1,\ldots$ (recall that we denote $N_j = 2^{2^j}$ for $j\geq 1$ and $N_0=1$). 
Then $\log(M_n+2)\leq 2^{n+1}$.
Notice that there are $|{\cal A}_n|\leq N_n$ points of the form $\pi_n(t)-\pi_{n-1}(t)$, $t\in T$.
So we may set $s^1:=0$ and for $n=1,2,\ldots$ define
$s^{k}$, $M_{n-1}<k\leq M_n$ as some rearrangement (with repetition if $|{\cal A}_n|<N_n$) 
of points of the form $(\pi_n(t)-\pi_{n-1}(t))/d_{2^{n+1}}(\pi_n(t),\pi_{n-1}(t))$, $t\in T$.
Then $\|X_{s^k}\|_{\log(k+2)}\leq 1$ for all $k$.

Observe that 
\[
\|t-\pi_n(t)\|_2=\|X_t-X_{\pi_n(t)}\|_{2}\leq \Delta_2(A_n(t))\leq \Delta_{2^n}(A_n(t))\rightarrow 0
\mbox{ for }n\rightarrow\infty.
\]
For any $s,t\in T$ we have $\pi_0(s)=\pi_0(t)$ and thus
\[
s-t=\lim_{n\rightarrow\infty}(\pi_n(s)-\pi_n(t))=
\lim_{n\rightarrow\infty}\left( \sum_{k=1}^n(\pi_k(s)-\pi_{k-1}(s))-\sum_{k=1}^n(\pi_k(t)-\pi_{k-1}(t))\right).
\]
This shows that
\[
T-T\subset R\ \overline{\mathrm{conv}}\{\pm s^k\colon\ k\geq 1\}, 
\]
where
\begin{align*}
R&:=2\sup_{t\in T}\sum_{n=1}^{\infty}d_{2^{n+1}}(\pi_n(t),\pi_{n-1}(t))
\leq 2\sup_{t\in T}\sum_{n=1}^{\infty}\Delta_{2^{n+1}}(A_{n-1}(t))
\\
&\leq  C(\alpha)\sup_{t\in T}\sum_{n=1}^{\infty}\Delta_{2^{n-1}}(A_{n-1}(t))
\leq C(\alpha,\beta)\Ex\sup_{s,t\in T}(X_s-X_t),
\end{align*}
where the second inequality follows by Lemma \ref{lem:convreg} and the last one by \eqref{eq:part}.
Thus it is enough to define $t^k:=Rs^k$, $k\geq 1$.
\end{proof}

\begin{proof}[Proof of Proposition \ref{prop:suptosud}]
Fix $p\geq 1$ and $T\subset \ell^2$ such that $|T|\geq e^p$ and $\|X_s-X_t\|_p\geq u$ for distinct points $s,t\in T$.
For $t^1,t^2\in T$ define a new point in $\ell^2$ by $t(t^1,t^2):=(t^1_1,t^2_1,t^1_2,t^2_2,\ldots)$.
Put also $\wt{T}:=\{t(t^1,t^2)\colon t^1,t^2\in T\}$. It is not hard to see that
$\|X_s-X_t\|_p\geq u$ for $t,s\in \wt{T}$, $t\neq s$. 

Choose an integer $k$ such that $2^k\leq  p< 2^{k+1}$ and let $({\cal A}_n)$ be an admissible sequence of partitions 
of the set $\wt{T}$. Since $|\wt{T}|=|T|^2\geq e^{2p}> 2^{2^{k+1}}$,
there is $A\in {\cal A}_k$ which contains at least two points of $\tilde{T}$.
Hence
\[
u\leq \Delta_{2^k}(A)\leq \gamma_X(\wt{T})\leq \frac{1}{\kappa} \Ex\sup_{s,t\in \wt{T}}(X_s-X_t)\leq 
\frac{2}{\kappa} \Ex\sup_{s,t\in T}(X_s-X_t).
\]
\end{proof}

\section*{Acknowledgements}

The second named author would like to thank Filip Borowiec for a fruitful discussion regarding Lemma \ref{lm:convex}.

\noindent
{\sc Rafa{\l} Lata{\l}a}\\
Institute of Mathematics\\
University of Warsaw\\
Banacha 2\\
02-097 Warszawa, Poland\\
\texttt{rlatala@mimuw.edu.pl}

\medskip
\noindent
{\sc Tomasz Tkocz}\\
Mathematics Institute\\
University of Warwick\\
Coventry CV4 7AL, UK\\
\texttt{t.tkocz@warwick.ac.uk}

\end{document}